\documentclass[12pt]{amsart}
\usepackage{amssymb,latexsym}
\usepackage{graphicx}
\newtheorem{theorem}{Theorem}
\newtheorem{lemma}[theorem]{Lemma}

\def\R{{\mathbb R}}

\begin{document}
\title[Regularity for matrices and sparse graphs]{Szemer\'edi's Regularity Lemma for matrices and sparse graphs}
\author{Alexander Scott}

\address{Mathematical Institute\\
    24-29 St Giles'\\
    Oxford, OX1 3LB\\
    UK;
    email: scott@maths.ox.ac.uk}

\date{\today}

\begin{abstract}
Szemer\'edi's Regularity Lemma is an important tool for analyzing the structure of dense graphs.  There are versions of 
the Regularity Lemma for sparse graphs, but these only apply when the graph satisfies some local density condition.  
In this paper, we prove a sparse Regularity Lemma that holds for all graphs.  More generally, we give a Regularity Lemma that holds for arbitrary real matrices.
\end{abstract}

\maketitle

\section{Introduction}\label{S:intro}

\subsection{Background}

Let $X$ and $Y$ be disjoint sets of vertices in $G$.  We say that the pair $(X,Y)$ is {\em $\epsilon$-regular} if, for every $X'\subset X$ and $Y'\subset Y$ with $|X'|\ge\epsilon|X|$ and $|Y'|\ge\epsilon|Y|$, we have
\begin{equation}\label{regd}
|d(X',Y')-d(X,Y)|<\epsilon,
\end{equation} 
where $d(X,Y)$ is the density between $X$ and $Y$ (for notation, see Section \ref{notation}).
Note that $\epsilon$ plays two roles here, bounding both the size of the subsets $X'$ and $Y'$ and the difference in density.  

The aim of Szemer\'edi's Regularity Lemma is to break up a graph into pieces such that the bipartite
graphs between different pieces are mostly `well-behaved' (i.e.~$\epsilon$-regular).
We shall therefore consider various partitions $V_0\cup\cdots\cup V_k$ of $V(G)$, 
often with a specified vertex class $V_0$, which we shall refer to as the {\em exceptional set}
(we will always use the subscript 0 for the exceptional set). 
A partition $V(G)=V_0\cup\cdots\cup V_k$ with exceptional set $V_0$ is {\em balanced} if $|V_i|=|V_j|$ for all $i,j\ge 1$.
We say that a partition $V(G)=V_0\cup\cdots\cup V_k$ with exceptional set $V_0$ is {\em $\epsilon$-regular} if it is balanced, $|V_0|<\epsilon|G|$ and all but at most $\epsilon k^2$ pairs $(V_i,V_j)$ with $i>j\ge1$ are $\epsilon$-regular (we will often suppress explicit mention of the exceptional set).  For a partition $\mathcal P$ with an exceptional set, we write $|\mathcal P|$ for the number of nonexceptional classes (so we ignore the exceptional set).

Szemer\'edi's Regularity Lemma \cite{S} then says the following.

\begin{theorem}[Szemer\'edi's Regularity Lemma]
For every $\epsilon>0$ and every integer $m\ge1$ there is an integer $M$ such that every graph $G$ with $|G|\ge M$ has an $\epsilon$-regular partition $\mathcal P$ with $|\mathcal P|\in[m,M]$.
\end{theorem}

Szemer\'edi's Regularity Lemma is an extremely important tool for analysing the structure of dense graphs.
However, for sparse graphs, it is much less helpful.
Indeed, if the graph does not contain a large set of vertices that induces a reasonably dense subgraph 
then {\em every} balanced partition (into not too many classes) is $\epsilon$-regular.  
It is therefore desirable to have a version of Szemer\'edi's Regularity Lemma that carries some form of structural information even
for graphs with very few edges
(for further background and discussion on regularity and sparse graphs see
Gerke and Steger \cite{GS} and Bollob\'as and Riordan \cite{BR}).

In order to handle sparse graphs, it is natural to modify the notion of regularity to take the density of the graph into account.
We say that a pair $(X,Y)$ is {\em $(\epsilon,p)$-regular} if, for every $X'\subset X$ and $Y'\subset Y$ with $|X'|\ge\epsilon|X|$ and $|Y'|\ge\epsilon|Y|$, we have
\begin{equation}\label{sregd}
|d(X',Y')-d(X,Y)|<\epsilon p.
\end{equation}
We say that $(X,Y)$ is {\em $(\epsilon)$-regular} if it is $(\epsilon,d)$-regular, where $d$ is the density of $G$.  A partition $V_0\cup\cdots\cup V_k$ with exceptional set $V_0$ is {\em $(\epsilon)$-regular} if it is balanced, $|V_0|<\epsilon|G|$, and all but at most $\epsilon k^2$ pairs $(V_i,V_j)$ with $i>j\ge1$ are $(\epsilon,d)$-regular.

Using this notion of regularity, 
Kohayakawa and R\"odl (see \cite{K,KR,GS})
proved a Sparse Regularity Lemma for a large class of sparse graphs, namely those that do not have large dense parts.
More precisely, we say that a graph with density $d$ is {\em $(\eta,D)$-upper-uniform}
if, for all disjoint $X, Y \subset V$ with $\min\{|X|,|Y|\}\ge \eta|G|$, we have 
$d(X,Y) \le Dd$.  

\begin{theorem}\label{KRthm}\ \!\cite{KR}
For every $\epsilon,D>0$ and every integer $m\ge1$ there are $\eta>0$ and an integer $M$ such that every 
$(\eta,D)$-upper uniform graph $G$ has an $(\epsilon)$-regular partition $\mathcal P$ with $|\mathcal P|\in[m,M]$. 
\end{theorem}

Random graphs typically satisfy an upper uniformity condition, which has meant that this result has been useful in practice.
However, it is natural to wonder whether the upper uniformity condition is necessary, or whether some simpler condition could replace it
(see for instance Gerke and Steger \cite{GS} and Koml\'os and Simonovits \cite{KS96} for discussion).
The aim of this paper is to prove that, perhaps surprisingly, no condition at all is necessary:
for every $\epsilon>0$ and $m\ge1$ there is an integer $M$ such that {\em every} sufficiently large graph has an $(\epsilon)$-regular partition
into $k$ parts, for some $k\in [m,M]$.  Furthermore, similar results hold for arbitrary weighted graphs and for real matrices.

\subsection{Results}
Let us note first that $(\epsilon)$-regularity can be considered as a `rescaled' version of $\epsilon$-regularity.
Consider $G$ as an edge-weighting $w_G$ of $K_n$, where $w(xy)=1$ for edges $xy\in E(G)$ and $w(xy)=0$ for nonedges.  Now, multiplying all weights by $\binom n2 /e(G)$, 
we get an edge-weighting $w'$ with average edge-weight 1 (note that the weights can be arbitrarily large, even if $w$ is bounded).  
An $(\epsilon)$-regular partition of $G$ with edge-weighting $w$ corresponds to an $\epsilon$-regular partition of $G$ with edge-weighting $w'$.

In proving our results, it will be both more general and more convenient to work in terms of matrices rather than graphs.  
Let $A=(a_{ij})$ be a matrix (not necessarily square), with rows indexed by $V$ and columns indexed by $W$.  We write $$||A||=\sum_{i\in V, j\in W}|a_{ij}|.$$
For $X\subset V$ and $Y\subset W$, we write
$$w_A(X,Y)=\sum_{v\in X,w\in Y}a_{vw}$$
and say that the {\em density} of the submatrix $A_{X,Y}$ (with rows $X$ and columns $Y$) is
$$d_A(X,Y)=\frac{w_A(X,Y)}{|X| |Y|}.$$
We say that a submatrix $A_{X,Y}$ is {\em $\epsilon$-regular} if for all $X'\subset X$ and $Y'\subset Y$ with $|X'|\ge\epsilon|X|$ and $|Y'|\ge\epsilon|Y|$, we have
$$|d_A(X',Y')-d_A(X,Y)|\le\epsilon.$$

A {\em block partition} $(\mathcal P,\mathcal Q)$ of $A$ is a partition $\mathcal P$ of $V$ together with a partition $\mathcal Q$ of $W$; 
the {\em blocks} $(X,Y)$ of the partition are the submatrices $A_{X,Y}$ for $X\in\mathcal P$ and $Y\in\mathcal Q$.    
If $V=W$ (i.e.~rows and columns are indexed by the same set),
we say that the block partition is {\em symmetric} if $\mathcal P=\mathcal Q$.
We will also want to allow exceptional sets, which for block partitions are given by specifying a class $V_0$ of $\mathcal P$ and a class $W_0$ of $\mathcal Q$.  
We say that the block partition $(\mathcal P,\mathcal Q)$ has {\em exceptional sets $(V_0,W_0)$}, and refer to the blocks $\{(V_0,Y):Y\in\mathcal Q\}\cup\{(X,W_0):X\in\mathcal P\}$ as {\em exceptional blocks}.
The block partition $(\mathcal P,\mathcal Q)$ with exceptional sets $(V_0,W_0)$ is {\em symmetric} if $\mathcal P=\mathcal Q$ and $V_0=W_0$.

We say that a block partition $(\mathcal P,\mathcal Q)$ with exceptional sets $(V_0,W_0)$ is {\em balanced} if the partitions $\mathcal P$ (with exceptional set $V_0$) and $\mathcal Q$ (with exceptional set $W_0$) are balanced.
A block partition $(\mathcal P,\mathcal Q)$ with exceptional sets $(V_0,W_0)$ is {\em $\epsilon$-regular} if $\mathcal P$  and $\mathcal Q$ (with exceptional sets $V_0$, $W_0$ respectively) are balanced, $|V_0|<\epsilon|V|$, $|W_0|<\epsilon|W|$, and all but at most $\epsilon|\mathcal P| |\mathcal Q|$ of the nonexceptional blocks are $\epsilon$-regular.  $(\mathcal P,\mathcal Q)$ is {\em $(\epsilon)$-regular} if it is an $\epsilon$-regular partition of the normalized matrix 
$$A^*=\frac{|V| |W|}{||A||} A$$ 
in which the average modulus of entries is 1.

It is not difficult to prove a version of the Regularity Lemma for matrices $A$ such that all entries of $A^*$ are $O(1)$ (or more generally, 
following Theorem \ref{KRthm}, for matrices satisying a suitable local density condition).
The difficulty  comes when the entries are not uniformly bounded.  Note that, for a graph $G$ with density $d$ and adjacency matrix $A$, the normalized matrix $A^*$ has entries with maximum value $\Theta(1/d)$; so for sparse graphs, the entries of $A^*$ can become arbitrarily large (for instance, consider a graph with $n$ vertices and $n\log n$ edges).

Our aim here is to prove a version of the Regularity Lemma for {\em arbitrary} weighted graphs and matrices.
For general matrices, we have the following result.

\begin{theorem}[SRL for matrices]\label{srlm}
For every $\epsilon>0$ and every positive integer $L$ there is a positive integer $M$ such that, for all $m,n\ge M$, every real $m$ by $n$ matrix $A$ has an  $(\epsilon)$-regular block partition $(\mathcal P,\mathcal Q)$ with $|\mathcal P|,|\mathcal Q|\in [L,M]$.
\end{theorem}

For square matrices, we can demand a little more.  

\begin{theorem}[SRL for square matrices]\label{srlsm}
For every $\epsilon>0$ and every positive integer $L$ there is a positive integer $M$ such that, for all $n\ge M$, every real $n$ by $n$ matrix $A$ has a symmetric $(\epsilon)$-regular partition $(\mathcal P,\mathcal P)$ with $|\mathcal P|\in [L,M]$.
\end{theorem}

Given a graph $G$, we can apply Theorem \ref{srlsm} to the adjacency matrix $A$ of $G$ (note that partitions of $V(G)$ correspond to {\em symmetric} partitions of $A$).  There is a rough correspondence between $(\epsilon)$-regular partitions of $G$ and symmetric, $(\epsilon)$-regular block partitions of the adjacency matrix $A(G)$: each gives an $(O(\epsilon))$-regular partition of the other (the slight difference is caused by the fact that diagonal blocks are relevant in a block partition $(\mathcal P,\mathcal P)$ but do not correspond to pairs in the partition $\mathcal P$ of $V(G)$; however this is easily handled by a slight rescaling of $\epsilon$).
Theorem \ref{srlsm} therefore has the following immediate corollary.

\begin{theorem}[SRL for sparse graphs]\label{srlsg}
For every $\epsilon>0$ and every positive integer $m$ there is a positive integer $M$ such that every graph $G$ with at least $M$ vertices has an $(\epsilon)$-regular partition $\mathcal P$ with $|\mathcal P|\in [m,M]$.
\end{theorem}

In fact, by the same argument, the result holds for graphs with an arbitrary edge-weighting.

Note that Theorem \ref{srlsg} is stronger than Theorem \ref{KRthm}, as it does not have the upper uniformity hypothesis.  However, some additional constraints may be needed in applications (see Section \ref{s4}).

\subsection{Notation}\label{notation}
Let $G$ be a graph.  For $X\subset V(G)$ we define $e_G(X)=e(G[X])$ to be the number of edges induced by $X$; for disjoint $X,Y$ we define $e(X,Y)=|\{xy\in E(G):x\in X, y\in Y\}|$.  For a set of vertices $X\subset V(G)$ the {\em density} of $X$ is $d(X)=e(X)/\binom{|X|}{2}$, and the {\em density of $G$} is $d(G)=d(V(G))$; for disjoint sets $X,Y\subset V(G)$ we define $d(X,Y)=e(X,Y)/(|X| |Y|)$.  

For integers $a<b$, we write $[a,b]=\{a,a+1,\ldots,b\}$.

We will write partitions either as $\mathcal P=V_1\cup\cdots\cup V_k$ or as $\mathcal P=\{V_0,\ldots,V_k\}$.
For a partition $\mathcal P=V_1\cup\cdots\cup V_k$ of a set $S$, we write $|\mathcal P|=k$;
for a partition $\mathcal P=V_0\cup\cdots\cup V_k$ with exceptional set $V_0$, we write $|\mathcal P|=k$, i.e.~we are counting the number of nonexceptional classes.  

\section{Averaging}

Given a matrix $A=(a_{xy})$ with rows and columns indexed by $V$ and $W$ respectively, we will be interested in various partitions of $A$ into blocks and in how the entries of $A$ are distributed with respect to these partitions. 

Let $\phi:\R\to\R$ be a function.
For subsets $X\subset V$ and $Y\subset W$, we define
$$
\phi(X,Y)=|X| |Y|\phi(d(X,Y)).
$$  
More generally, if ${\mathcal P}=V_1\cup\cdots\cup V_k$ is a partition of some subset of $V$
and ${\mathcal Q}=W_1\cup\cdots\cup W_l$ is a partition of some subset of $W$, we define
$$\phi(\mathcal P,\mathcal Q)=\sum_{i,j}\phi(V_i,W_j).$$  

This can also be seen in terms of ``smoothing'' or ``averaging''.
For partitions ${\mathcal P}=V_1\cup\cdots\cup V_k$ of $V$
and ${\mathcal Q}=W_1\cup\cdots\cup W_l$ of $W$
we define the matrix $A_{\mathcal P,\mathcal Q}$ by (for every $i,j$)
$$
(A_{\mathcal P,\mathcal Q})_{xy}=
d(V_i,W_j)\quad \text{for $x\in V_i$ and $y\in W_j$.}
$$
Thus $A_{\mathcal P,\mathcal Q}$ is the matrix obtained by averaging the entries of $A$ inside each block $(V_i,W_j)$. If $\mathcal P$ is a partition of some subset $X$ of $V$ and $\mathcal Q$ is a partition of some subset $Y$ of $W$ then we define
$$A_{\mathcal P,\mathcal Q}=(A_{X,Y})_{\mathcal P,\mathcal Q}.$$
Note that if $\mathcal P'$ and $\mathcal Q'$ are refinements of $\mathcal P$ and $\mathcal Q$ then
\begin{equation}\label{tower}
 (A_{\mathcal P',\mathcal Q'})_{\mathcal P,\mathcal Q}=A_{\mathcal P,\mathcal Q}.
\end{equation}
This follows by a trivial calculation (and is a special case of the tower law for conditional expectation).

For a real-valued function $\phi$, 
we can now define
$$\phi(A)=\sum_{x\in V,y\in W}\phi(a_{xy}).$$
Thus, writing $d_{ij}=d(V_i,W_j)$, we have
\begin{equation}\label{phidef}
\phi(\mathcal P,\mathcal Q)
=\sum_{i,j}\phi(d_{ij})|V_i| |W_j|
=\phi(A_{\mathcal P,\mathcal Q}).
\end{equation}

Note that $\phi$ is additive.  For instance, given partitions $\mathcal P=V_1\cup\cdots\cup V_k$ and $\mathcal Q=W_1\cup\cdots\cup W_l$ of (subsets of) $V$ and $W$, and refinements $\mathcal P'$ of $\mathcal P$ and $\mathcal Q'$ of $\mathcal Q$, we can write $\mathcal P_i'$, $\mathcal Q_j'$ for the partitions of $V_i$ and $W_j$ induced by $\mathcal P'$, $\mathcal Q'$ respectively.  Then 
$$\phi(\mathcal P',\mathcal Q')=\sum_{i,j}\phi(\mathcal P_i',\mathcal Q_j').$$

We have not said how $\phi(\mathcal P,\mathcal Q)$ is defined for partitions with exceptional sets.  If $\mathcal P=V_0\cup\cdots\cup V_k$ is a partition with exceptional set $V_0$, we define $\widetilde{\mathcal P}$ to be the partition obtained from $\mathcal P$ by splitting $V_0$ into singletons.  If $(\mathcal P,\mathcal Q)$ has exceptional sets $(V_0,W_0)$, we define $\phi(\mathcal P,\mathcal Q)$ by
\begin{equation}\label{phiexcept}
\phi(\mathcal P,\mathcal Q)=\phi(\widetilde{\mathcal P},\widetilde{\mathcal Q}).
\end{equation}

\section{Proofs}

The usual proof of the Regularity Lemma proceeds roughly as follows:
\begin{itemize}
\item  Define a function $f$ on partitions $\mathcal P$ of $G$ such that $0\le f(\mathcal P)\le 1$ for any partition
\item  Show that if a partition $\mathcal P$ is not $\epsilon$-regular then there is a balanced refinement $\mathcal Q$ of $\mathcal P$ so that $|\mathcal Q|$ is bounded (as a function of $|\mathcal P|$) and $f(\mathcal Q)\ge f(\mathcal P)+\alpha$, where $\alpha=\alpha(\epsilon)$ is a fixed constant
\item  Iterating the step above, we get a sequence of partitions, each refining the previous one and increasing the value of $f$ by a constant.  Since $f$ is nonnegative and bounded above by 1, the process cannot have more than $1/\alpha$ iterations, and so we must terminate with an $\epsilon$-regular partition with a bounded (but possibly very large) number of classes.
\end{itemize}

We take the same general approach.  However, there is a significant obstacle in the proof.
The standard function to use for partitions $\mathcal P=V_0\cup\cdots\cup V_k$ is
\begin{equation}\label{standard}
 f(\mathcal P)=\frac{1}{n^2}\sum_{1\le i<j\le k}d(V_i,V_j)^2|V_i| |V_j|.
\end{equation}
Since all densities $d(V_i,V_j)$ are at most 1, it is easily seen that $f(\mathcal P)\le1$ for any partition $\mathcal P$.
However, 
for arbitrary weighted graphs or matrices, this is no longer true: for instance, in a weighted graph we can have average weight 1 but almost all the mass 
concentrated on a small number of edges, which can make $f(\mathcal P)$ arbitrarily large.

An important feature of our argument is that, rather than using a quadratic function as in \eqref{standard}, we can instead use a different convex function 
(this approach has previously been used by \L uczak \cite{L00}).
We will therefore work with a different function $\phi(x)$, which is quadratic for small values of $x$ but has a ``cutoff'' after which it becomes linear, to prevent large entries making a disproportionate contribution to $\phi(\mathcal P,\mathcal Q)$.  
It is also convenient to work with the (slightly) 
greater generality of a matrix rather than an edge-weighted graph; of course, as noted already, we can readily translate between the two by considering the (weighted) 
adjacency matrix.  Finally, for convenience, we leave out the $1/n^2$ factor from our weight functions.

\subsection{Effects of refinement}
We start by looking at the effects on $\phi(\mathcal P,\mathcal Q)$ of refining the partitions $\mathcal P$ and $\mathcal Q$. We first show that, provided $\phi$ is convex, refinements do not decrease $\phi(\mathcal P,\mathcal Q)$.

\begin{lemma}\label{monotone}
Let $A=(a_{xy})$ be a real matrix with rows indexed by $V$ and columns indexed by $W$. Let $\phi:\R\to\R$ be a convex function.
Suppose that $(\mathcal P,\mathcal Q)$ is a block partition of $A$ and $(\mathcal P',\mathcal Q')$ is a refinement of $(\mathcal P,\mathcal Q)$.  Then
\begin{equation*}
\phi(\mathcal P', \mathcal Q')\ge \phi(\mathcal P,\mathcal Q).
\end{equation*}
\end{lemma}

\begin{proof}
In light of \eqref{tower} and \eqref{phidef}, this follows easily from Jensen's Inequality for conditional expectation.  However, we give a proof for completeness.

Suppose first that $\mathcal P=\{V\}$ and $\mathcal Q=\{W\}$ are the trivial partitions, and $\mathcal P'=X_1\cup\cdots\cup X_r$ and $\mathcal Q'=Y_1\cup\cdots\cup Y_s$.  Let $d=d(V,W)$ and $d_{ij}=d(X_i,Y_j)$, so (since the total weight is conserved)
$$d|V| |W|=\sum_{i,j}d_{ij}|X_i| |Y_j|.$$
Since $\phi$ is convex, and $\sum_{i,j}|X_i| |Y_j|=|V| |W|$, we have
$$\phi(d)\le\sum_{i,j}\frac{|X_i| |Y_j|}{|V| |W|}\phi(d_{ij}),$$
and so
\begin{equation}\label{eqla}
\phi(\mathcal P,\mathcal Q)=\phi(d)|V| |W|\le \sum_{i,j}\phi(d_{ij})|X_i| |Y_j|=\phi(\mathcal P',\mathcal Q').
\end{equation}
Now suppose $\mathcal P=V_1\cup\cdots\cup V_k$ 
and $\mathcal Q=W_1\cup\cdots\cup W_l$.  Writing $\mathcal P_i'$ and $\mathcal Q_j'$ respectively for the partitions of $V_i$ and $W_j$ induced by $\mathcal P'$ and $\mathcal Q'$, we have by \eqref{eqla} that
$
\phi(\mathcal P,\mathcal Q)
=\sum_{i,j}\phi(V_i,W_j)
\le\sum_{i,j}\phi(\mathcal P_i',\mathcal Q_j')\\
=\phi(\mathcal P', \mathcal Q')$.
\end{proof}

Note that Lemma \ref{monotone} also applies to partitions with exceptional sets: if $\mathcal P$ is a partition with exceptional set $V_0$ 
and $\mathcal P'$ has exceptional set $V_0'$, we say that $\mathcal P'$ is a {\em refinement} of $\mathcal P$ if
$V_0'\supset V_0$, and every other element of $\mathcal P'$ is contained in some element of $\mathcal P$.  This notion of refinement extends naturally to block partitions $(\mathcal P, \mathcal Q)$ with exceptional sets $(V_0,W_0)$, by considering $\mathcal P$ and $\mathcal Q$ separately.
Then Lemma \ref{monotone} is then easily seen to apply: if $(\mathcal P',\mathcal Q')$ with exceptional sets $(V_0',W_0')$ is a refinement of $(\mathcal P,\mathcal Q)$ with exceptional sets $(V_0,W_0)$, then
$$\phi(\mathcal P', \mathcal Q')\ge \phi(\mathcal P,\mathcal Q).$$

\subsection{Finding a good refinement} 
In order to obtain quantitative bounds on the effects of refinement, we must choose a specific convex function for $\phi$.
Let us fix $\epsilon>0$ and $D>0$, and define the function $\phi=\phi_{\epsilon,D}$ by
$$
\phi(t)=
\begin{cases}
t^2 & \text{if $|t|\le 2D$}\\
4D(|t|-D) & \text{otherwise}.
\end{cases}
$$
Note that $\phi$ is convex and $\phi(t)\le 4D|t|$ for all $t$.  We will choose a value $D=D(\epsilon)$ later (in order to clarify the presentation, we carry constants through the lemmas below).

We note the following trivial bound on $\phi(A)$.

\begin{lemma}\label{phiupperbound}
Let $A=(a_{xy})$ be a matrix, and
let $(\mathcal P,\mathcal Q)$ be a block partition of $A$. Then
$$\phi(\mathcal P,\mathcal Q)\le 4D ||A||.$$
\end{lemma}

\begin{proof}
Let $\mathcal P^*$ and $\mathcal Q^*$ be the partitions of $V$ and $W$ respectively into singletons.
Then, by Lemma \ref{monotone}, we have
$$\phi(\mathcal P,\mathcal Q)\le\phi(\mathcal P^*,\mathcal Q^*)=\sum_{i,j}\phi(a_{ij})\le4D\sum_{i,j}|a_{ij}|=4D||A||.$$
\end{proof}

Note that the definition \eqref{phiexcept} implies that the same result applies to block partitions with exceptional sets.

Lemma \ref{monotone} shows that refinements do not decrease $\phi$; we will also need to find refinements that increase $\phi$ by a significant amount.  This will
be a consequence of the following lemma.

\begin{lemma}\label{smallepsgain}
Suppose $D\ge1$, $\epsilon\in(0,1/2)$, and 
$A$ is a matrix with rows indexed by $V$ and columns indexed by $W$.
Supose that $X\subset V$ and $Y\subset W$.
If $|d(X,Y)|\le\epsilon D$ and $(X,Y)$ is not $\epsilon$-regular then there are partitions
$\mathcal X=X_1\cup X_2$ of $X$ and $\mathcal Y=Y_1\cup Y_2$ of $Y$ such that 
\begin{equation}\label{gain}
\phi(\mathcal X, \mathcal Y)\ge \phi(X,Y)+\epsilon^4|X| |Y|.
\end{equation}
\end{lemma}

\begin{proof}
Since $(X,Y)$ is not $\epsilon$-regular, we can find sets $X_1\subset X$ and $Y_1\subset Y$ such that $|X_1|\ge\epsilon|X|$, $|Y_1|\ge\epsilon|Y|$ and $|d(X_1,Y_1)-d(X,Y)|\ge\epsilon$.  Now if $|X_1|>|X|/2$ then we can replace $X_1$ by a subset $X_1'\subset X_1$ with $|X_1'|=\lfloor |X|/2\rfloor$ and $|d(X_1',Y_1)-d(X,Y)|\ge\epsilon$ (consider a random $X_1'\subset X_1$ of this size).  Thus we may assume that $\epsilon|X|\le|X_1|\le|X|/2$, and similarly $\epsilon|Y|\le|Y_1|\le|Y|/2$.  Let $X_2=X\setminus X_1$ and $Y_2=Y\setminus Y_2$, and note that 
$$\min\{|X_1|,|X_2|\}\ge\epsilon|X|, \quad \min\{|Y_1|,|Y_2|\}\ge\epsilon|Y|.$$
Let $\mathcal X=X_1\cup X_2$ and $\mathcal Y=Y_1\cup Y_2$.

Now if any pair $(X_i,Y_j)$ has $|d(X_i,Y_j)|\ge2D$ then 
\begin{align*}
\phi(\mathcal X,\mathcal Y)
&\ge\phi(d_{ij})|X_i| |Y_j|\\
&\ge\phi(2D)\cdot\epsilon|X|\cdot\epsilon|Y|\\
&\ge 4\epsilon^2D^2|X| |Y|.
\end{align*}
But $\phi(X,Y)\le\phi(\epsilon D)|X| |Y|=\epsilon^2D^2|X| |Y|$ and $D\ge1$, so 
$$\phi(\mathcal X,\mathcal Y)>\phi(X,Y)+\epsilon^4|X| |Y|.$$

Otherwise, all pairs $(X_i,Y_j)$ have density at most $2D$.  
Let $d=d(X,Y)$ and $d_{ij}=d(X_i,Y_j)$.
Now $\sum_{i,j}d_{ij}|X_i| |Y_j|=d|X| |Y|$, so 
defining $\eta_{ij}$ by $d_{ij}=d+\eta_{ij}$, we have 
\begin{equation}\label{eta0}
\sum_{i,j}\eta_{ij}|X_i| |Y_j|=0.
\end{equation}
It follows that
\begin{align*}
\phi(\mathcal X,\mathcal Y)
&=\sum_{i,j}\phi(d_{ij})|X_i| |Y_j|\\
&=\sum_{i,j}(d+\eta_{ij})^2|X_i| |Y_j|\\
&=d^2|X| |Y|+2d\sum_{i,j}\eta_{ij}|X_i| |Y_j|+\sum_{i,j}\eta_{ij}^2|X_i| |Y_j|\\
&\ge d^2|X| |Y|+\eta_{11}^2|X_1| |Y_1|\\
&\ge\phi(X,Y)+\epsilon^4|X| |Y|,
\end{align*}
where we have used \eqref{eta0} in the penultimate line, and the fact that $|\eta_{11}|\ge\epsilon$ in the final line.
\end{proof}

The main lemma in our argument is as follows. 

\begin{lemma}\label{mainl}
Suppose that $\epsilon\in(0,1/2)$, $D\ge 8/\epsilon^2$, and  $A$ is a matrix with rows indexed by $V$ and columns indexed by $W$ such that $||A||=|V| |W|$.
Suppose that $\mathcal P=V_0\cup\cdots\cup V_k$ is a balanced partition of $V$ with exceptional set $V_0$,
and
$\mathcal Q=W_0\cup\cdots\cup W_l$ is a balanced partition of $W$ with exceptional set $W_0$, where
$$|V_0|<|V|/2, \qquad|W_0|<|W|/2$$
and 
$$|V|\ge k4^{l+1}, \qquad |W|\ge l4^{k+1}.$$ 
If the block partition $(\mathcal P, \mathcal Q)$ with exceptional sets $(V_0,W_0)$ is not $\epsilon$-regular then there is a balanced refinement 
$(\mathcal P',\mathcal Q')$ of $(\mathcal P,\mathcal Q)$ with excptional sets $(V_0',W_0')$, such that
\begin{align}
|\mathcal P'|&\le k4^{l+1}, \quad |\mathcal Q'|\le l4^{k+1}\notag\\
|V_0'|&\le|V_0|+\frac{|V|}{2^l},\quad |W_0'|\le|W_0|+\frac{|W|}{2^k}\label{exc}
\end{align}
and
\begin{equation}\label{globalinc}
\phi(\mathcal P',\mathcal Q')\ge\phi(\mathcal P,\mathcal Q)+\frac{\epsilon^5|V| |W|}{8}.
\end{equation}
\end{lemma}

\begin{proof}
Let $n=|V|$ and $m=|W|$.
Note first that, for $i,j\ge1$, we have $|V_i|\ge n/2k$ and $|W_l|\ge m/2l$.
Since $(\mathcal P,\mathcal Q)$ is not $\epsilon$-regular, there are at least $\epsilon kl$ blocks $(V_i,W_j)$ with $i,j\ge1$ such that $(V_i,W_j)$ is not $\epsilon$-regular.  Now at most $\epsilon kl/2$ of these blocks have $|d(V_i,W_j)|\ge \epsilon D$, or else we would have
$$||A||>(\epsilon kl/2)(n/2k)(m/2l)\epsilon D=\epsilon^2 mnD/8\ge mn.$$
So there are at least $\epsilon kl/2$ irregular blocks with $|d(V_i,W_j)|<\epsilon D$.

For each pair $(i,j)$, $i,j\ge1$, we define partitions $\mathcal C_{ij}$ of $V_i$ and $\mathcal D_{ij}$ of $W_j$ as follows.
\begin{itemize}
\item  If $(V_i,W_j)$ is $\epsilon$-regular or $|d(V_i,W_j)|\ge \epsilon D$, we take the trivial partitions $\mathcal C_{ij}=\{V_i\}$ and $\mathcal D_{ij}=\{W_j\}$.
\item  If $(V_i,W_j)$ is not $\epsilon$-regular and $|d(V_i,W_j)|<\epsilon D$, we can find partitions of $\mathcal C_{ij}$ of $V_i$ and $\mathcal D_{ij}$ of $W_j$ into two sets each, as in Theorem \ref{smallepsgain}, so that \eqref{gain} is satisfied.
\end{itemize}
Now for $i,j\ge1$, we let $\mathcal P_i$ be the partition of $V_i$ generated by $\{\mathcal C_{ij}:j\ne i\}$
and $\mathcal Q_j$ be the partition of $W_j$ generated by $\{\mathcal D_{ij}:i\ne j\}$.  
Let $\mathcal P_0$ be the partition of $V_0$ into singletons
and let $\mathcal Q_0$ be the partition of $W_0$ into singletons.  
Let ${\mathcal P^*}$ be the partition of $V$ obtained by concatenating the $\mathcal P_i$, $i\ge 0$, 
and let ${\mathcal Q^*}$ be the partition of $W$ obtained by concatenating the $\mathcal Q_i$, $i\ge 0$.  Then
$$\phi({\mathcal P^*},{\mathcal Q^*})
=\phi(\mathcal P_0,\mathcal Q_0)
+\sum_{i\ge 1}\phi(\mathcal P_0,\mathcal Q_i)
+\sum_{i\ge 1}\phi(\mathcal P_i,\mathcal Q_0)
+\sum_{i,j\ge1}\phi(\mathcal P_i,\mathcal Q_j).$$
On the other hand, recalling \eqref{phiexcept}, we have
$$\phi(\mathcal P,\mathcal Q)=\phi(\mathcal P_0,\mathcal Q_0)
+\sum_{i\ge 1}\phi(\mathcal P_0,W_i)
+\sum_{i\ge 1}\phi(V_i,\mathcal Q_0)
+\sum_{i,j\ge1}\phi(V_i,W_j).$$
Now for all $i,j\ge1$,
\begin{align*}
\phi(\mathcal P_0,\mathcal Q_j)\ge \phi(\mathcal P_0,W_j),\\
\phi(\mathcal P_i,\mathcal Q_0)\ge \phi(V_i,\mathcal Q_0),\\
\phi(\mathcal P_i,\mathcal Q_j)\ge \phi(V_i,W_j),
\end{align*}
and, in $\epsilon kl/2$ cases,
\begin{equation}\label{ekl}
\phi(\mathcal P_i,\mathcal Q_j)\ge \phi(V_i,W_j)+\epsilon^4|V_i| |W_j|,
\end{equation}
where the first inequality follows from Lemma \ref{monotone} and the second one from \eqref{gain}.
Since $|V_i| |W_j|\ge mn/4kl$, we have
\begin{align*}
\phi(\mathcal P^*,\mathcal Q^*)
&\ge\phi({\mathcal P},{\mathcal Q})
+\frac{\epsilon kl}{2}\cdot\epsilon^4\cdot\frac{mn}{4kl}\\
&=\phi(\mathcal P,\mathcal Q)+\frac{\epsilon^5mn}{8}.
\end{align*}

We are almost done, except the partitions $\mathcal P^*$ and $\mathcal Q^*$ may not be balanced.  
We start with exceptional set $V_0$, which contains all sets from $\mathcal P_0$, and
divide each remaining set in $\mathcal P^*$ into subsets of size $\lfloor n/k4^l\rfloor$, throwing away (i.e.~adding to the exceptional set) any remainder.  Since each nonexceptional set in $\mathcal P$ is divided into at most $2^l$ pieces in $\mathcal P^*$, and we throw away at most $n/k4^l$ vertices from each set in $\mathcal P^*$, the exceptional set increases in size by at most
$$\frac{n}{k4^l}\cdot k2^l=\frac{n}{2^l}$$
vertices.  Let the resulting partition be $\mathcal P'$, and let $V_0'$ be the resulting exceptional set.  
Note that $|\mathcal P'|\le n/\lfloor n/k4^l\rfloor\le k4^{l+1}$.

We similarly divide $\mathcal Q^*$ into pieces of size $m/l4^k$, increasing the size of the exceptional set by at most $m/2^k$.  The resulting block partition $(\mathcal P',\mathcal Q')$ with exceptional sets $(V_0',W_0')$ is a refinement of $(\mathcal P,\mathcal Q)$ (with exceptional sets $(V_0,W_0)$), and still satisfies \eqref{globalinc} by Lemma \ref{monotone}. 
\end{proof}

\subsection{Final proofs}
After all this, it is straightforward to prove our main theorems.  

\begin{proof}[Proof of Theorem \ref{srlm}]
We may assume that $||A||=|V| |W|$, or else replace $A$ by the normalized matrix $A^*$.  Let $D=8/\epsilon^2$, so by Lemma \ref{phiupperbound} we have $\phi(\mathcal P',\mathcal Q')\le32|V| |W|/\epsilon^2$ for any block partition $(\mathcal P,\mathcal Q)$ of $A$.

We start with any balanced block partition $(\mathcal P,\mathcal Q)$ with $|\mathcal P|=|\mathcal Q|=\max\{L,\lceil\log(1/\epsilon)\rceil+2\}$, 
and with exceptional sets $(V_0,W_0)$ such that $|V_0|,|W_0|<L$.  We now apply Lemma \ref{mainl} repeatedly: at each iteration, we either have an $\epsilon$-regular partition, or can apply the lemma to obtain a partition for which $\phi$ increases by at least $\epsilon^5|V| |W|/8$ and such that the increase in size of the exceptional sets is controlled by \eqref{exc}.  Since $\phi(\mathcal P',\mathcal Q')$ is always at most $32|V| |W|/\epsilon^2$, we halt after at most $256/\epsilon^7$ steps with an $\epsilon$-regular partition.
\end{proof}

\begin{proof}[Proof of Theorem \ref{srlsm}]
We proceed as in the proof of Theorem \ref{srlm}.  However, we begin by choosing a balanced block partition $(\mathcal P,\mathcal Q)$ with $|\mathcal P|=|\mathcal Q|=4L/\epsilon$, so that there are at most $\epsilon|\mathcal P| |\mathcal Q|/4$ diagonal blocks.

We now apply a slight modification of Lemma \ref{mainl}: in order to keep a symmetric partition, we choose $\mathcal C_{ij}=\mathcal D_{ji}$ at every stage.   
For $i=j$, we choose the trivial partition $\mathcal C_{ii}=\mathcal D_{ii}=\{V_i\}$.
For $i\ne j$, we choose partitions for $(V_i,V_j)$ and $(V_j,V_i)$ at the same time, and choose the same partition up to transposition for each: note that if
$A$ is not symmetric and both $(V_i,V_j)$ and $(V_j,V_i)$ are irregular, then we may have to choose partitions such that \eqref{ekl} holds 
for only one of $(V_i,V_j)$ and $(V_j,V_i)$.
However, since there are at most $\epsilon kl/4$ diagonal blocks, 
there are at least 
$\epsilon kl/4$ irregular off-diagonal blocks, and so
equation \eqref{ekl} still applies at least $\epsilon kl/8$ times.  Therefore $\phi(\mathcal P,\mathcal Q)$ increases by at least $\epsilon^5mn/32$ at each step.
\end{proof}

Note that, in the case that $A$ is the adjacency matrix of a graph $G$, if a diagonal block is $\epsilon$-regular we obtain some pseudorandomness inside 
the corresponding class of the partition; however, we are allowed to fail on a proportion $\epsilon$ of blocks, 
so we can fail to be $\epsilon$-regular on all diagonal blocks.

\section{Final remarks}\label{s4}

Although Theorem \ref{srlsg} seems to give a natural version of Szemer\'edi's Regularity Lemma for sparse graphs, significant difficulties remain in applications.  
For instance, the usual counting and embedding lemmas that are invaluable in the dense case are not immediately useful in the sparse case.  Furthermore, as edge weights may be unbounded 
(after rescaling), Theorem \ref{srlsg} may give a partition in which all the edges of the original graph are hidden in irregular pairs. However, 
even in this case it might be helpful to know that the edges are confined to some small part of the graph.

The results above are stated for graphs and matrices.  It is straightforward to write down versions for digraphs, for coloured graphs, 
or for several matrices with arbitrary real entries.  For instance the proof of Theorem \ref{srlm} is easily modified to prove the following 
(by applying Lemma \ref{mainl} to each matrix in turn at each stage, and refining a common block partition).

\begin{theorem}\label{srlk}
For every $\epsilon\in(0,1)$ and positive integers $L,k$ there is a positive integer $M$ such that for all $m,n\ge M$ and
for every sequence $A^{(1)},\ldots,A^{(k)}$ of $m$ by $n$ matrices with nonnegative entries
there is a  block partition $(\mathcal P,\mathcal Q)$ with $|\mathcal P|,|\mathcal Q|\in [L,M]$
that is simultaneously $(\epsilon)$-regular for each $A_i$.
\end{theorem}

Once again, in the symmetric case, we can demand $\mathcal P=\mathcal Q$.

Finally, we note that taking $D=\infty$ in the arguments above gives a particularly clean version of the usual proof of Szemer\'edi's Regularity Lemma.

\medskip

\noindent{\bf Acknowledgement.}  The author would like to thank Paul Balister for very helpful suggestions, and a referee for pointing out the paper of \L uczak \cite{L00}.


\begin{thebibliography}{9}

\bibitem{BR}
B.~Bollob\'as and O.~Riordan, Metrics for sparse graphs, {\em in} Surveys in combinatorics 2009, 
211--287, London Math.~Soc.~Lecture Note Ser.~{\bf 365}, Cambridge Univ. Press, Cambridge, 2005.

\bibitem{GS}
S.~Gerke and A.~Steger, The sparse regularity lemma and its applications, 
{\em in} Surveys in combinatorics 2005, 227--258, London Math.~Soc.~Lecture Note Ser.~{\bf 327}, 
Cambridge Univ. Press, Cambridge, 2005.

\bibitem{K}
Y.~Kohayakawa, Szemer\'edi's regularity lemma for sparse graphs, {\em in} 
Foundations of Computational Mathematics, Rio de Janeiro, 1997, Springer, Berlin, 1997, pp. 216--230.

\bibitem{KR}
Y.~Kohayakawa, V.~R\"odl, Szemer\'edi's regularity lemma and quasirandomness, 
{\em in} Recent Advances in Algorithms and Combinatorics, CMS Books Math./Ouvrages Math. SMC 11, Springer, New York (2003), 289--351.

\bibitem{KS96}
J.~Koml\'os and M.~Simonovits,
Szemer\'edi's regularity lemma and its applications in graph theory,
{\em in} Combinatorics, Paul Erdős is eighty, Vol.~2 (Keszthely, 1993), 295--352, 
Bolyai Soc.~Math.~Stud. {\bf 2}, J\'anos Bolyai Math.~Soc., Budapest, 1996.

\bibitem{L00}
T.~\L uczak, 
On triangle-free random graphs,
{\em Random Structures and Algorithms} {\bf 16} (2000), 260--276.

\bibitem{S}
E.~Szemer\'edi, Regular partitions of graphs, {\em in} Probl\`emes combinatoires et th\'eorie des graphes, 
Colloq. Internat. CNRS, Univ. Orsay, Orsay, 1976, CNRS, Paris, 1978, pp. 399--401.

\end{thebibliography}
\end{document}